\documentclass{article}
\usepackage[english]{babel}
\usepackage[utf8]{inputenc}
\usepackage{amsmath}
\usepackage{graphicx}
\usepackage{amsfonts}
\usepackage{enumitem}
\usepackage{amssymb}
\usepackage[colorinlistoftodos]{todonotes}
\usepackage{geometry}
\usepackage{amsthm}

\newcommand{\norm}[1]{\left\lVert#1\right\rVert}

\newcommand{\paren}[1]{\left( #1 \right)}

\newcommand{\set}[1]{\left\{ #1 \right\}}
\newcommand{\setcond}[2]{\left\{ #1 \;\middle\vert\; #2 \right\}}

\newcommand{\RR}{\mathbb{R}}

\newcommand{\NN}{\mathbb{N}}

\newtheorem{thm}{Theorem}[section]
\newtheorem{lem}[thm]{Lemma}

\newtheorem{conj}[thm]{Conjecture}

\theoremstyle{definition}
\newtheorem{defn}[thm]{Definition}

\title{Difference sets in $\RR^d$}
\author{David Conlon\thanks{Department of Mathematics, Caltech, Pasadena, CA 91125, USA. Email: {\tt dconlon@caltech.edu}. Research supported by NSF Award DMS-2054452.} \and Jeck Lim\thanks{Department of Mathematics, Caltech, Pasadena, CA 91125, USA. Email: {\tt jlim@caltech.edu}. Research partially supported by the NUS Overseas Graduate Scholarship.} }
\date{}

\begin{document}
\maketitle

\begin{abstract}
Let $d \geq 2$ be a natural number. We show that 
$$|A-A| \geq \left(2d-2 + \frac{1}{d-1}\right)|A|-(2d^2-4d+3)$$ 
for any sufficiently large finite subset $A$ of $\RR^d$ that is not contained in a translate of a hyperplane. By a construction of Stanchescu, this is best possible and thus resolves an old question first raised by Uhrin.
\end{abstract}

\section{Introduction}

Given two subsets $A, B$ of an abelian group, the sumset $A + B$ is defined by
\[A + B = \{a + b : a \in A, b \in B\}\]
and the difference set $A - B$ is defined similarly. 
One of the fundamental results in additive combinatorics is Freiman's structure theorem, the statement that any finite set of integers $A$ with small doubling, that is, with $|A + A| \leq K|A|$ for some fixed constant $K$, is contained in a generalised arithmetic progression of small size and dimension. The first step in Freiman's original proof~\cite{F73} of this theorem is a simple lemma showing that if 
$A$ is a finite $d$-dimensional subset of $\RR^d$, then
$$|A+A| \geq (d+1)|A| - d(d+1)/2,$$
where we say that a subset $A$ of $\RR^d$ is \emph{$k$-dimensional} and write $\dim(A) = k$ if the dimension of the affine subspace spanned by $A$ is $k$. Freiman's result is tight, as may be seen by considering the union of $d$ parallel arithmetic progressions with the same common difference.

Surprisingly, the analogous problem of estimating $|A-A|$ for $d$-dimensional subsets $A$ of $\RR^d$ has remained open, despite first being raised by Uhrin~\cite{U80} in 1980 because of connections to the geometry of numbers and then reiterated many times (see, for example,~\cite{CL07, FHU90, R94, S98, S01}). 
However, the first few cases are well understood. Indeed, for $d = 1$, it is an elementary observation that $|A-A| \geq 2|A|-1$, which is tight for arithmetic progressions, while, for $d = 2$, the bound $|A-A| \geq 3|A|-3$, tight for the union of two parallel arithmetic progressions with the same length and common difference, was proven by Freiman, Heppes and Uhrin~\cite{FHU90}. More generally, they showed that if $A$ is a finite $d$-dimensional subset of $\RR^d$, then
$$|A-A| \geq (d+1)|A| - d(d+1)/2,$$
in analogy with Freiman's result on $|A+A|$. This estimate was later generalised by Ruzsa~\cite{R94}, who showed that if $A, B \subset \RR^d$ are finite sets such that $|A|\geq|B|$ and $\dim(A + B) = d$, then 
\begin{align}
    |A + B| \geq |A| + d|B| - d(d + 1)/2. \label{eqn:ruzsa}
\end{align}
Finally, for $d = 3$, Stanchescu~\cite{S98}, making use of this inequality of Ruzsa, proved that $|A-A| \geq 4.5|A| - 9$ for any finite $3$-dimensional subset $A$ of $\RR^3$. This is again tight, with the example now being a parallelogram of four parallel arithmetic progressions with the same length and common difference. 

For higher dimensions, the best known construction is due to Stanchescu~\cite{S01} and comes from a collection of $2d-2$ carefully placed parallel arithmetic progressions with the same length and common difference. More precisely, set $T = \{e_0, e_1, \dots, e_{d-2}\}$, where $e_0$ is the origin and $\{e_1, \dots, e_d\}$ is the standard basis for $\RR^d$, and, for any natural number $k$, let $A_k = (T \cup (a_k - T)) + P_k$, 
where $a_k = e_d - k e_{d-1}$ and $P_k = \{e_0, e_{d-1}, 2e_{d-1}, \dots, (k-1) e_{d-1}\}$. Worked out carefully, this construction satisfies
$$|A_k - A_k| = \left(2d - 2+\frac{1}{d-1}\right)|A_k| - (2d^2-4d+3).$$
Supplanting an earlier conjecture of Ruzsa~\cite{R94}, 
Stanchescu proposed that this is best possible. 

\begin{conj}[Stanchescu~\cite{S01}] \label{conj:stan}
Suppose $d \geq 2$ and $A \subset \RR^d$ is a finite set such that $\dim(A) = d$. Then 
$$|A - A| \geq \left(2d - 2+\frac{1}{d-1}\right)|A| - (2d^2-4d+3).$$
\end{conj}

Until very recently, little was known about this conjecture for $d \geq 4$ besides the result of Freiman, Heppes and Uhrin~\cite{FHU90}. However, the situation was considerably improved by Mudgal~\cite{M20}, who showed that 
$$|A - A| \geq (2d - 2)|A| - o(|A|)$$
for any finite $d$-dimensional subset $A$ of $\RR^d$. 
Our main result, which builds on both Mudgal's work and earlier work of Stanchescu~\cite{S98, S10}, is a proof of Conjecture~\ref{conj:stan} in full provided only that $|A|$ is sufficiently large in terms of $d$, essentially resolving the problem of minimising the value of $|A-A|$ over all $d$-dimensional sets $A$ of a given size.

\begin{thm} \label{thm:main}
Suppose $d \geq 2$ and $A \subset \RR^d$ is a finite set such that $\dim(A) = d$. Then, provided $|A|$ is sufficiently large in terms of $d$, 
$$|A - A| \geq \left(2d - 2+\frac{1}{d-1}\right)|A| - (2d^2-4d+3).$$
\end{thm}

We begin our proof of Theorem~\ref{thm:main} in the next section with a result that we believe to be of independent interest, an extension of a result of Stanchescu~\cite{S10} about the structure of $d$-dimensional subsets $A$ of $\RR^d$ with doubling constant smaller than $d + 4/3$ to asymmetric sums $A + B$.

\section{An asymmetric version of a theorem of Stanchescu}

Our starting point is with the following theorem of Stanchescu~\cite{S10} (see also~\cite{S08} for the $d = 3$ case).

\begin{thm}[Stanchescu~\cite{S10}] \label{thm:Stan}
Suppose $d \geq 2$ and $A \subset \mathbb{R}^d$ is a finite set with $\dim(A) = d$. If $|A| > 3 \cdot 4^{d}$ and $|A + A| < (d + 4/3)|A| - \frac{1}{6}(3d^2 +5d + 8)$, then $A$ can be covered by $d$ parallel lines.
\end{thm}

By considering the set $A = A_0 \cup \{e_3, \dots, e_d\}$ with $A_0 = \{i e_1 + j e_2 : 0 \leq i < n, 0 \leq j \leq 2\}$ for some natural number $n$, which satisfies $|A+A| = (d + 4/3)|A| - \frac{1}{6}(3d^2 +5d + 8)$ and yet cannot be covered by $d$ parallel lines, we see that Theorem~\ref{thm:Stan} is tight. The main result of this section is an extension of Theorem~\ref{thm:Stan} to asymmetric sums $A + B$. We begin with the two-dimensional case, whose proof relies in a critical way on the following result of Grynkiewicz and Serra~\cite[Theorem 1.3]{GS10}.

\begin{lem}[Grynkiewicz--Serra~\cite{GS10}] \label{lem:gs}
Let $A,B\subset \RR^2$ be finite sets, let $l$ be a line, let $r_1$ be the number of lines parallel to $l$ which intersect $A$ and let $r_2$ be the number of lines parallel to $l$ that intersect $B$. Then
$$|A + B| \geq \paren{\frac{|A|}{r_1}+\frac{|B|}{r_2}-1}(r_1+r_2-1).$$
\end{lem}

In particular, we note that, since $|B| \geq r_2$ and $r_1 \geq 1$,
$$|A + B| \geq \frac{r_2}{r_1}|A|.$$

\begin{lem} \label{lem:base}
Let $A,B\subset\RR^2$ be finite sets and $l$ be a fixed line. Let $r_1$ be the number of lines parallel to $l$ which intersect $A$. If $|A|\geq |B|$ and $|A+B|<|A|+7|B|/3-5\sqrt{|A|}$, then either $r_1\leq 2$ or $r_1> |A|/4$.
\end{lem}

\begin{proof}
Notice that if $A$ is at most 1 dimensional, then either $r_1=1$ or $r_1=|A|$, so we may assume that $\dim(A)=2$. Let $r_2$ be the number of lines parallel to $l$ which intersect $B$. 
We consider 2 cases, depending on whether $r_1$ is at most $\sqrt{|A|}$ or not.\\

\noindent
\underline{Case 1: $r_1\leq \sqrt{|A|}$}

\noindent
 We have $10|A|/3\geq |A+B|\geq |A|r_2/r_1$, 
 so $r_2\leq 10r_1/3\leq 4\sqrt{|A|}$. Thus, by Lemma~\ref{lem:gs} and the fact that $|A| \geq |B|$,
\begin{align*}
    |A+B| &\geq \paren{\frac{|A|}{r_1}+\frac{|B|}{r_2}-1}(r_1+r_2-1)\\
    &= |A|+\frac{r_2-1}{r_1}|A|+\paren{1+\frac{r_1-1}{r_2}}|B|-r_1-r_2+1\\
    &\geq |A|+\paren{1+\frac{r_2-1}{r_1}+\frac{r_1-1}{r_2}}|B|-5\sqrt{|A|}.
\end{align*}
If $r_2=1$ and $r_1 \geq 3$, then this last expression is $|A|+r_1|B|-5\sqrt{|A|}\geq |A|+3|B|-5\sqrt{|A|}$. If $r_2=2$ and $r_1 \geq 3$, then it is
$$|A|+\paren{\frac12+\frac{1}{r_1}+\frac{r_1}{2}}|B|-5\sqrt{|A|}\geq |A|+\frac73 |B|-5\sqrt{|A|}.$$
If $r_2\geq 3$ and $r_1 \geq 3$, then it is at least
$$|A|+\paren{3-\frac{1}{r_1}-\frac{1}{r_2}}|B|-5\sqrt{|A|}\geq |A|+\frac73 |B|-5\sqrt{|A|}.$$ 
In each case, we contradict our assumption that $|A+B|<|A|+7|B|/3-5\sqrt{|A|}$, so we must have $r_1 \leq 2$.\\

\noindent
\underline{Case 2: $r_1\geq \sqrt{|A|}$}

\noindent
 Let $r_1'=|A|/r_1$ and  $r_2'=|B|/r_2$, so that $r_1'\leq \sqrt{|A|}$ and
$$|A + B| \geq \paren{\frac{|A|}{r_1'}+\frac{|B|}{r_2'}-1}(r_1'+r_2'-1),$$
which is the same expression as in the previous case, but now $r_1',r_2'$ may not be integers. Nevertheless, we still have $1\leq r_1'\leq |A|$ and $1\leq r_2'\leq |B|$, so that $|A+B|\geq \frac{r_2'}{r_1'}|A|$ and, therefore, $r_2'\leq 4\sqrt{|A|}$ holds similarly. 
Expanding the equation above and using $|A| \geq |B|$, we have
\begin{align*}
    |A+B| &\geq |A|+\paren{1+\frac{r_2'}{r_1'}+\frac{r_1'-1}{r_2'}-\frac{1}{r_1'}}|B|-5\sqrt{|A|}\\
    &\geq |A|+\paren{1+2\sqrt{\frac{r_1'-1}{r_1'}}-\frac{1}{r_1'}}|B|-5\sqrt{|A|}.
\end{align*}
Setting $c=\sqrt{\frac{r_1'-1}{r_1'}}$, we see that if $r_1 \leq |A|/4$ or, equivalently, $r'_1 \geq 4$, then $c\geq \frac{\sqrt3}{2}$ and the expression above is $|A|+(2c+c^2)|B|-5\sqrt{|A|}\geq |A|+7|B|/3-5\sqrt{|A|}$. But this again contradicts our assumption, so we must have $r_1 > |A|/4$.
\end{proof}

For higher dimensions, we will use an induction scheme based on taking a series of compressions. Let us first say what a compression is in this context.

\begin{defn}
Let $H$ be a hyperplane in $\RR^d$ and $v\in\RR^d$ a vector not parallel to $H$. For a finite set $A\subset \RR^d$, the \emph{compression of $A$ onto $H$ with respect to $v$}, denoted by $P(A)=P_{H,v}(A)$, 
is formed by replacing the points on any line $l$ parallel to $v$ which intersects $A$ at $s \geq 1$ points with the points $u+jv$, $j=0,1,\ldots,s-1$, where $u$ is the intersection of $l$ with $H$.
\end{defn}

By preserving the ordering of the points on each line, we may view the compression $P$ as a pointwise map $A\to P(A)$, so we may talk about points of $A$ being fixed by $P$. Note that it is clearly the case that $|P(A)|=|A|$. Moreover, sumsets cannot increase in size after applying this compression operation. That this is the case is our next result. 

\begin{lem} \label{lem:compress}
For finite sets $A,B\subset\RR^d$ and a compression $P$,
$$|P(A)+P(B)|\leq |A+B|.$$
\end{lem}

\begin{proof}
Without loss of generality, we may assume that $H$ passes through the origin. Let $p:\RR^d\to H$ be the projection onto $H$ along $v$. For $u\in p(A)$, let $l_u$ be the line through $u$ parallel to $v$ and define $X_u=X\cap l_u$ for any set $X\subset\RR^d$. Note that $p(P(A))=p(A)$ and so $p(P(A)+P(B))=p(A+B)$. It therefore suffices to show that  $|(P(A)+P(B))_u|\leq |(A+B)_u|$ for each $u\in p(A+B)=p(A)+p(B)$. Since $P(A)_x$ is a set of the form $\setcond{x+jv}{j=0,\ldots,s-1}$, we have 
\begin{align*}
    |(P(A)+P(B))_u| &=\max\setcond{|P(A)_x+P(B)_y|}{x\in p(A),y\in p(B),x+y=u}\\
    &= \max\setcond{|P(A)_x|+|P(B)_y|-1}{x\in p(A),y\in p(B),x+y=u}\\
    &= \max\setcond{|A_x|+|B_y|-1}{x\in p(A),y\in p(B),x+y=u}\\
    &\leq |(A+B)_u|. \qedhere
\end{align*}
\end{proof}

Our main compression lemma, which draws on ideas in the work of Stanchescu~\cite{S08, S10}, is now as follows.

\begin{lem} \label{lem:reduce}
 Let $A,B\subset \RR^d$ be finite sets such that $\dim(A)=d\geq 3$ and $l$ be a fixed line. Suppose that there are exactly $s<|A|$ lines parallel to $l$ which intersect $A$. Then there are sets $A',B'\subset \RR^d$ satisfying the following properties:
\begin{enumerate}
    \item $|A'|=|A|$, $|B'|=|B|$;
    \item $|A'+B'|\leq |A+B|$;
    \item there are exactly $s$ lines $l_1',\ldots,l_s'$ parallel to $l$ intersecting $A'$;
    \item $\dim(A')=d$;
    \item $l_1',\ldots,l_{s-1}'$ lie on a hyperplane;
    \item $l_s'$ intersects $A'$ at a single point.
\end{enumerate}
\end{lem}

\begin{proof}
The sets $A',B'$ will be obtained by taking a series of compressions, so 1 and 2 will automatically be satisfied by Lemma~\ref{lem:compress}. Let $e_1,\ldots,e_d$ be the standard basis of $\RR^d$. By applying an affine transformation if necessary, we may assume that $l$ is the line $\RR e_d$ and that $A$ contains the set $S=\set{0,e_1,\ldots,e_d}$ (this is possible since at least one line parallel to $l$ intersects $A$ in at least 2 points). 
For each $i$, let $H_i$ be the hyperplane through 0 perpendicular to $e_i$. Let $P_i=P_{H_i,e_i}$ be the compression onto $H_i$ with respect to $e_i$. Let $A_1=P_d(A)$, 
noting that this set satisfies 3 and $s=|A_1\cap H_d|$. Furthermore, for any compression $P_i$, $i<d$, $|P_i(A_1)\cap H_d|=s$, so $P_i(A_1)$ also satisfies 3. Now set $A_2=P_1(P_2(\cdots P_{d-1}(A_1)\cdots ))$. 
Then $A_2 \subset \NN_0^d$  again satisfies 3 and, since $S\subseteq A_2$, $\dim(A_2)=d$ and it also satisfies 4. Moreover, $A_2$ has the property that if $(x_1,\ldots,x_d)\in A_2$, then, for any $y_1,\ldots,y_d\in \NN_0$ with $y_i\leq x_i$ for all $i$, $(y_1,\ldots,y_d)\in A_2$.

We now show that a finite number of further compressions will give us a set additionally satisfying 5 and 6. Suppose $A_2$ can be covered by $n$ hyperplanes parallel to $H_{d-1}$, i.e., the $(d-1)$th coordinate of all the points of $A_2$ is the set $\set{0,1,\ldots,n-1}$. Let $w=(w_1,\ldots,w_{d-2},0,0)\in A_2$ be such that $w_1+\cdots +w_{d-2}$ is maximal. Then, whenever $tw+u\in A_2\cap H_{d-1}\cap H_d$ for some $u\in \NN_0^d$ and $t\geq 1$, we must have $u=0$ and $t=1$. Let $P$ be the compression onto $H_{d-1}$ with respect to $f=e_{d-1}-w$. Set $A_3=P(A_2)$. Since $f$ is parallel to $H_d$, $|A_3\cap H_d|=|A_2\cap H_d|=s$. The number of lines through $A_3$ parallel to $l$ is $|A_3\cap H_d|=s$, so 3 is still satisfied. Moreover, since $w\in A_2$, $e_{d-1}$ is fixed by $P$, so $S\subseteq A_3$ and 4 is still satisfied. We now consider two cases:\\

\noindent
\underline{Case 1: $n=2$}

\noindent
We claim that $A_3$ is covered by $H_{d-1}$ and the single line $e_{d-1}+\RR e_d$, so that 5 is satisfied with $l_s'=e_{d-1}+\RR e_d$. Indeed, by the maximality of $\norm{w}_1$, the points of $A_2$ on any vertical line $u+\RR e_d$ with $u\in H_d\setminus\set{e_{d-1}}$ are mapped by $P$ into a vertical line contained in $H_{d-1}$. To see this, suppose $e_{d-1}+re_d+v\in A_2$ with $v\in H_{d-1}\cap H_d$ and $r\in \NN_0$. Then $e_{d-1}+re_d+v$ is fixed 
by $P$ iff $v+re_d+w\in A_2$. If $v\neq 0$, then $v+w\not\in A_2$ by the maximality of $w$, so $v+re_d+w\not\in A_2$ and $e_{d-1}+re_d+v$ is not fixed by the compression, being moved instead to $v+re_d+w$. \\

\noindent
\underline{Case 2: $n>2$}

\noindent
Suppose $(n-1)e_{d-1}+v\in A_2$ with $v\in H_{d-1}$. Then, since $(n-1)w+v\not\in A_2$ as in Case 1, $(n-1)e_{d-1}+v$ is not fixed by the compression. Thus, $A_3$ is contained in fewer than $n$ hyperplanes parallel to $H_{d-1}$. By repeatedly applying compressions of this type, we will eventually reach the previous case. Abusing notation very slightly, we shall still call the set obtained after these repeated compressions $A_3$.\\

Thus, $A_3$ is covered by $H_{d-1}$ and the line $e_{d-1}+\RR e_d$. Suppose now that $r>0$ is the largest integer such that $re_d\in A_3$. Let $P'$ be the compression with respect to $g=e_{d-1}-re_d$ and set $A_4=P'(A_3)$. Then all points of $A_3$ in $H_{d-1}$ and $e_{d-1}$ are fixed by $P'$, but $e_{d-1}+te_d$ is mapped to $(r+t)e_d$ for each $t>0$. Thus, $A_4\cap (e_{d-1}+H_{d-1})=\set{e_{d-1}}$, so that $A_4$ 
satisfies 3-6. We may therefore set $A' = A_4$. Finally, to obtain $B'$, we simply apply the same series of compressions to $B$ that we applied to $A$.
\end{proof}

We are now in a position to prove the main result of this section, the promised asymmetric version of  Theorem~\ref{thm:Stan}.

\begin{thm} \label{thm:asym}
Let $d\geq 2$, $A,B\subset \RR^d$ be finite sets and $l$ be a line. Let $r$ be the number of lines parallel to $l$ which intersect $A$. Suppose that $A$ is $d$-dimensional, $|A|\geq |B|$ and $|A+B|<|A|+(d+1/3)|B|-2^{d+1}\sqrt{|A|}-E_d$, where $E_d=(d+2)^{2^d-2}$. Then $r=d$ or $r>|A|/4$.
\end{thm}

\begin{proof}
Notice that since $\dim(A)=d$, we must have $r\geq d$. We shall induct on $d$. The case $d=2$ was dealt with in Lemma~\ref{lem:base}. We may therefore assume that $d\geq 3$. $E_d$ is chosen to satisfy the following inequalities:
\begin{enumerate}
    \item $E_d\geq 2(E_{d-1}+1)$,
    \item $E_d\geq (d+2)(2^d+E_{d-1}+1)^2$.
\end{enumerate}
If $|A|\leq (2^d+E_{d-1}+1)^2$, then $|A|+(d+1/3)|B|\leq (d+2)|A|\leq E_d$, so it is not possible that $|A+B|<|A|+(d+1/3)|B|-2^{d+1}\sqrt{|A|}-E_d$. We may therefore assume that $|A| > (2^d+E_{d-1}+1)^2$ and, thus, that  $|A|-2^d\sqrt{|A|}-E_{d-1}-1\geq 0$.

Suppose that $d<r\leq |A|/4$. By Lemma~\ref{lem:reduce}, replacing $A$ with $A'$, we can assume that $A=A_1\cup \{e_d\}$, where $A_1$ lies on the hyperplane $H$ defined by $x_d=0$. Let $H_1,\ldots,H_s$ be the hyperplanes parallel to $H$ that intersect $B$ and let $B_i=B\cap H_i$.

If $s=1$, then $|A+B|=|A_1+B|+|B|$. Moreover, $A_1$ is $(d-1)$-dimensional and is covered by $r-1\leq |A_1|/4$ lines parallel to $l$. Thus, if $|B|\leq |A_1|$, our induction hypothesis implies that $|A_1+B|\geq |A_1|+(d-1+1/3)|B|-2^d\sqrt{|A_1|}-E_{d-1}$. If instead $|B|>|A_1|$, then $|B|=|A_1|+1$, so, letting $B'$ be $B$ with an element removed, our induction hypothesis implies that  $|A_1+B|\geq |A_1+B'|\geq |A_1|+(d-1+1/3)(|B|-1)-2^d\sqrt{|A_1|}-E_{d-1}$. In either case, we have
\begin{align*}
    |A+B| &\geq |A_1|+(d+1/3)(|B|-1)-2^d\sqrt{|A_1|}-E_{d-1}\\
    &\geq |A|+(d+1/3)|B|-2^{d+1}\sqrt{|A|}-E_d.
\end{align*}

If $s\geq 2$, then $|A+B|\geq |A_1+B|= |A_1+B_1|+\cdots+|A_1+B_s|$. By our induction hypothesis, $|A_1+B_i|\geq |A_1|+(d-1+1/3)|B_i|-2^d\sqrt{|A_1|}-E_{d-1}$ for each $i$ and so 
\begin{align*}
    |A+B| &\geq s|A_1|+(d-1+1/3)|B|-2^d s\sqrt{|A_1|}-s E_{d-1}\\
    &\geq 2|A|+(s-2)|A|-s+(d-1+1/3)|B|-2^{d+1}\sqrt{|A|}-2^d(s-2)\sqrt{|A|}-sE_{d-1}\\
    &\geq |A|+(d+1/3)|B|-2^{d+1}\sqrt{|A|}-2(E_{d-1}+1)+(s-2)(|A|-2^d\sqrt{|A|}-E_{d-1}-1)\\
    &\geq |A|+(d+1/3)|B|-2^{d+1}\sqrt{|A|}-E_d. \qedhere
\end{align*}
\end{proof}

\section{Special cases of Theorem~\ref{thm:main}}

In this section, we show that the conclusion of 
Theorem~\ref{thm:main} holds if we make some additional assumptions about the structure of $A$. We begin with a simple example of such a result.

\begin{lem} \label{lem:dlines}
Let $A\subset \RR^d$ be a finite set with $\dim(A) = d$ that can be covered by $d$ parallel lines. Then
$$|A-A|\geq \paren{2d-2+\frac{2}{d}}|A|-(d^2-d+1).$$
\end{lem}

\begin{proof}
Suppose $A=A_1\cup \cdots\cup A_d$ where each $A_i$ lies on a line parallel to some fixed line $l$. Let $a_i=|A_i|$ and assume, without loss of generality, that $a_1\geq a_2\geq\cdots\geq a_d$. Since $A$ is $d$-dimensional, the $d$ lines covering $A$ are in general position, i.e., no $k$ of them lie on a $(k-1)$-dimensional affine subspace for each $1 \leq k \leq d$. Thus, for $i\neq j$, the sets $A_i-A_j$ are pairwise disjoint and also disjoint from $A_1-A_1$. Hence, we have
\begin{align*}
    |A-A| &\geq |A_1-A_1|+\sum_{i\neq j} |A_i-A_j|\\
    &\geq 2a_1-1+\sum_{i\neq j} (a_i+a_j-1)\\
    &\geq 2a_1-1+2(d-1)\sum_i a_i-d(d-1)\\
    &\geq \paren{2d-2+\frac{2}{d}}|A|-(d^2-d+1). \qedhere
\end{align*}
\end{proof}

We will use a common framework for the next two lemmas, with the following definition playing a key role.

\begin{defn}
Let $A\subset\RR^d$ be a finite set with $\dim(A) = d$ and $l$ be a fixed line. A hyperplane $H$ is said to be a \emph{supporting hyperplane} of $A$ if all points of $A$ either lie on $H$ or on one side of $H$. A supporting hyperplane $H$ of $A$ is said to be a \emph{major hyperplane of $A$ (with respect to $l$)} if $H$ is parallel to $l$ and $|H\cap A|$ is maximal.
\end{defn}

Suppose now that $A\subset \RR^d$ is $d$-dimensional and $l$ is a fixed line. 
Let $H$ be a major hyperplane with respect to $l$ and $H_1=H,H_2,\ldots,H_r$ be the hyperplanes parallel to $H$ that intersect $A$, arranged in the natural order. Let $A_i=A\cap H_i$ for $i=1,\ldots,r$. Since $|A_1|$ is maximal, $|A_1|\geq |A_r|$. Let $\pi$ be the projection along $l$ onto a hyperplane perpendicular to $l$. Then $\dim(\pi(A))=d-1$ and 
$\pi(H)$ is a maximal face of the convex hull of $\pi(A)$ (since $|H\cap A|$ is maximal), so $\dim(\pi(A_1))=d-2$, which implies that there are at least $d-1$ lines parallel to $l$ intersecting $A_1$. If any such line intersects $A_1$ in at least 2 points, then $\dim(A_1)=d-1$. Assuming this setup, the next lemma explores the situation where $A$ is covered by two parallel hyperplanes.

\begin{lem} \label{lem:2planes}
Suppose that $r=2$, $\dim(A_1)=d-1$ and there are $s$ lines parallel to $l$ intersecting $A_1$.
\begin{enumerate}
    \item If $s=d-1$, then
    \begin{align*}
        |A-A| &\geq \paren{2d-2}|A|+\frac{2}{d-1}|A_1|-(2d^2-4d+3)\\
        &\geq \paren{2d-2+\frac{1}{d-1}}|A|-(2d^2-4d+3).
    \end{align*}
    \item If $d\leq s\leq |A_1|/4$ and
    $$|A_1-A_1|\geq \paren{2d-4+\frac{1}{d-2}}|A_1|-(2d^2-8d+9),$$
    then, given 
    $0<\epsilon<\min(\frac23, \frac{1}{d-2})-\frac{1}{d-1}$, there is some $n_0$ such that for $|A|\geq n_0$,
    $$|A-A|\geq \paren{2d-2+\frac{1}{d-1}+\epsilon}|A|.$$
\end{enumerate}
\end{lem}

\begin{proof}
For  1, note, by Lemma~\ref{lem:dlines}, that
    $$|A_1-A_1|\geq\paren{2d-4+\frac{2}{d-1}}|A_1|-(d^2-3d+3).$$
    By Ruzsa's inequality~\eqref{eqn:ruzsa}, $|A_1-A_2|\geq |A_1|+(d-1)|A_2|-d(d-1)/2$ and so
    \begin{align*}
        |A-A| &\geq |A_1-A_1|+2|A_1-A_2|\\
        &\geq \paren{2d-2+\frac{2}{d-1}}|A_1|+(2d-2)|A_2|-d(d-1)-(d^2-3d+3)\\
        &\geq (2d-2)|A|+\frac{2}{d-1}|A_1|-(2d^2-4d+3)\\
        &\geq \paren{2d-2+\frac{1}{d-1}}|A|-(2d^2-4d+3).
    \end{align*}
    
    For 2, $A_1$ is $(d-1)$-dimensional and cannot be covered by $d-1$ lines, so this case only exists for $d\geq 3$. Since $|A_1|\geq |A_2|$,  Theorem~\ref{thm:asym} implies that
    $$|A_1-A_2|\geq |A_1|+(d-2/3)|A_2|-2^d\sqrt{|A_1|}-E_{d-1}.$$
    But then, since $|A_1|\geq |A|/2$ can be taken sufficiently large, 
    \begin{align*}
        |A-A| &\geq |A_1-A_1|+2|A_1-A_2|\\
        &\geq \paren{2d-4+\frac{1}{d-2}}|A_1|-(2d^2-8d+9)+2|A_1|+2(d-2/3)|A_2|-2^{d+1}\sqrt{|A_1|}-2E_{d-1}\\
        &\geq \paren{2d-2+\frac{1}{d-1}+\epsilon}|A|+\paren{\frac{1}{(d-1)(d-2)}-\epsilon}|A_1|-(2d^2-8d+9)-2^{d+1}\sqrt{|A_1|}-2E_{d-1}\\
        &\geq \paren{2d-2+\frac{1}{d-1}+\epsilon}|A|,
    \end{align*}
    as required.
\end{proof}

We now consider the situation where every line parallel to $l$ meets $A$ in a reasonable number of points.

\begin{lem} \label{lem:lines}
Let $0<\epsilon<1/(4d+1)(d-1)$. Suppose that every line parallel to $l$  intersecting $A$ intersects $A$ in at least $4d$ points. Then there is a constant $C_d$ such that either
\begin{enumerate}
    \item 
    $$|A-A|\geq \paren{2d-2+\frac{1}{d-1}+\epsilon}|A|-C_d$$ 
    or
    \item $r=2$ and
    $$|A-A|\geq \paren{2d-2}|A|+\frac{2}{d-1}|H\cap A|-(2d^2-4d+3).$$
\end{enumerate}
In particular,
$$|A-A|\geq \paren{2d-2+\frac{1}{d-1}}|A|-(2d^2-4d+3)$$
for $|A|$ sufficiently large.
\end{lem}

\begin{proof}
We shall induct on $d$ and $|A|$. Let $n_0$ be chosen sufficiently large that the following conditions hold:
\begin{enumerate}
    \item Lemma~\ref{lem:2planes} holds with this $n_0$.
    \item Whenever $B\subset \RR^d$ has $\dim(B)=d-1>1$, each line parallel to $l$ intersecting $B$ intersects it in at least $4(d-1)$ points  and $|B|\geq n_0/2$, then
    $$|B-B|\geq\paren{2d-4+\frac{1}{d-2}}|B|-(2d^2-8d+9).$$
    This is possible by induction since $C_{d-1}$ is already determined.
    \item $\epsilon n_0\geq d(d-1)$.
\end{enumerate}
Then $C_d\geq 2d^2-4d+3$ is chosen sufficiently large that the first option in the lemma trivially holds for $|A|\leq n_0$.

The base case $d=2$ and the inductive step will be handled together. If $|A|\leq n_0$, the lemma holds, so we may assume that $|A|>n_0$. Since $\dim(A_1)=d-1$, there are at least $d-1$ lines parallel to $l$ intersecting $A_1$. 
Each such line intersects $A_1$ in at least $4d$ points, so we have $|A_1|\geq 4d(d-1)$.

First suppose $r=2$. If $A_1$ is covered by $s$ lines parallel to $l$, then, as above, $s\geq d-1$. If $s=d-1$, then, by Lemma~\ref{lem:2planes},
$$|A-A|\geq (2d-2)|A|+\frac{2}{d-1}|A_1|-(2d^2-4d+3).$$
If $s>d-1$, then we must have $d>2$, since, for $d=2$, $\dim(A_1)=1$ and  $A_1$ is covered by a single line. 
Since $\dim(A_1)=d-1 > 1$ and $|A_1|\geq |A|/2\geq n_0/2$, condition 2 implies that
$$|A_1-A_1|\geq\paren{2d-4+\frac{1}{d-2}}|A_1|-(2d^2-8d+9).$$
Each line parallel to $l$ passes through at least 4 points of $A_1$, so $s\leq |A_1|/4$. Thus, by Lemma~\ref{lem:2planes} and condition 1,
$$|A-A| \geq \paren{2d-2+\frac{1}{d-1}+\epsilon}|A|.$$

Now suppose $r>2$. Let $B=A\setminus H_r$ and note that $\dim(B)=d$ and $|B|\geq |A|/2$. By our induction hypothesis,
$$|B-B|\geq \paren{2d-2+\frac{1}{d-1}}|B|-C_d.$$
Let $H'$ be a major hyperplane of $B$ with respect to $l$ (which is not necessarily a major hyperplane of $A$!), so that $|B\cap H'|\geq |A_1|$. If $|A_1|\geq 2\epsilon |A|$, then, using Ruzsa's inequality~\eqref{eqn:ruzsa} and condition $3$,
\begin{align*}
    |A-A| &\geq |B-B|+2|A_1-A_r|\\
    &\geq \paren{2d-2+\frac{1}{d-1}}|B|-C_d+2|A_1|+(2d-2)|A_r|-d(d-1)\\
    &\geq \paren{2d-2+\frac{1}{d-1}}|A|+\paren{2-\frac{1}{d-1}}|A_1|-C_d-d(d-1)\\
    &\geq \paren{2d-2+\frac{1}{d-1}+2\epsilon}|A|-C_d-d(d-1)\\
    &\geq \paren{2d-2+\frac{1}{d-1}+\epsilon}|A|-C_d.
\end{align*}
We may therefore assume that $|A_1|<2\epsilon |A|$.

If $B$ cannot be covered by two translates of $H'$, then, by our induction hypothesis,
$$|B-B|\geq \paren{2d-2+\frac{1}{d-1}+\epsilon}|B|-C_d.$$
Thus, again using Ruzsa's inequality~\eqref{eqn:ruzsa},
\begin{align*}
    |A-A| &\geq |B-B|+2|A_1-A_r|\\
    &\geq \paren{2d-2+\frac{1}{d-1}+\epsilon}|B|+2|A_1|+(2d-2)|A_r|-d(d-1)-C_d\\
    &\geq \paren{2d-2+\frac{1}{d-1}+\epsilon}|A|+\paren{2-\frac{1}{d-1}-\epsilon}|A_1|-d(d-1)-C_d\\
    &\geq \paren{2d-2+\frac{1}{d-1}+\epsilon}|A|-C_d,
\end{align*}
since $|A_1|\geq 4d(d-1)$.

We may therefore assume that $B$ is covered by two translates of $H'$, say $H'$ and $H''$. If $A_r\subseteq H'\cup H''$, then $A\subseteq H'\cup H''$, so one of $|A\cap H'|,|A\cap H''|$ is at least $|A|/2$, say $|A\cap H'|\geq |A|/2$. But $H$ is a major hyperplane of $A$, so $|A_1|=|A\cap H|\geq |A\cap H'|\geq |A|/2$, contradicting our assumption that $|A_1|<2\epsilon |A|$. Hence, $A_r\not\subseteq H'\cup H''$. 

If 
$$|B-B|\geq \paren{2d-2+\frac{1}{d-1}+\epsilon}|B|-C_d,$$
then the above argument holds similarly. Thus, by our induction hypothesis, we must have that
$$|B-B|\geq \paren{2d-2}|B|+\frac{2}{d-1}|H'\cap B|-(2d^2-4d+3).$$
Let $B_1=B\cap H',B_2=B\cap H''$, noting that $|B_1|\geq |B_2|$. Fix also a point $x\in A_r$ that does not lie on $H'\cup H''$. If $x$ lies between $H'$ and $H''$, then $x-B_1,B_1-x,B-B$ are pairwise disjoint. If $H'$ lies between $x$ and $H''$, then $x-B_2,B_2-x,B-B$ are pairwise disjoint. 
If $H''$ lies between $x$ and $H'$, then $x-B_1,B_1-x,B-B$ are pairwise disjoint. In any case, there is some $i\in\set{1,2}$ such that $x-B_i,B_i-x,B-B$ are pairwise disjoint. Since $|B_1|\geq |B_2|$,
\begin{align*}
    |A-A| &\geq |B-B|+2|B_2|\\
    &\geq \paren{2d-2}|B|+\frac{2}{d-1}|B_1|-(2d^2-4d+3)+2|B_2|\\
    &\geq \paren{2d-2+\frac{2}{d-1}}|B|-(2d^2-4d+3)\\
    &= \paren{2d-2+\frac{2}{d-1}}(|A|-|A_r|)-(2d^2-4d+3)\\
    &\geq \paren{2d-2+\frac{1}{d-1}+\epsilon}|A|-C_d,
\end{align*}
where the last inequality follows from $|A_r|\leq |A_1|\leq 2\epsilon |A|$ and $\epsilon<1/(4d+1)(d-1)$.
\end{proof}

\section{Proof of Theorem~\ref{thm:main}}

The final ingredient in our proof is the following 
 structure theorem due to Mudgal~\cite[Lemma 3.2]{M19}, saying that sets with small doubling in $\RR^d$ can be almost completely covered by a reasonably small collection of parallel lines.

\begin{lem}[Mudgal~\cite{M19}] \label{lem:mudgal}
For any $c > 0$, there exist constants $0 < \sigma \leq 1/2$ and $C > 0$ such that if $A \subset \RR^d$ is a finite set with $|A| = n$ and $|A + A| \leq cn$, then there exist parallel lines $l_1, l_2, \ldots, l_r$ with
$$|A \cap l_1| \geq \cdots \geq |A \cap l_r| \geq |A \cap l_1|^{1/2} \geq C^{-1}n^\sigma$$
and
$$|A \setminus (l_1 \cup l_2 \cup \cdots \cup l_r)| < Ccn^{1-\sigma}.$$
\end{lem}

We are now ready to prove Theorem~\ref{thm:main}, which, we recall, states that if $d \geq 2$ and $A \subset \RR^d$ is a finite set such that $\dim(A) = d$, then, provided $|A|$ is sufficiently large,
$$|A - A| \geq \paren{2d - 2+\frac{1}{d-1}}|A| - (2d^2-4d+3).$$

\begin{proof}[Proof of Theorem~\ref{thm:main}]
We shall proceed by induction on $d$, starting from the known case $d=2$~\cite{FHU90}. We will suppose throughout that $n_0$ is large enough for our arguments to hold. Our aim is to show that, for all $A \subset \RR^d$ with $\dim(A) = d$,
$$|A - A| \geq \paren{2d - 2+\frac{1}{d-1}}|A| - \max(2d^2-4d+3, D-|A|/3),$$
where $D\geq 2d^2-4d+3$ is chosen so that the above inequality trivially holds for $|A|\leq n_0$. The result then clearly follows for $|A|$ sufficiently large. We will proceed by induction on $|A|$, where the base case $|A|\leq n_0$ trivially holds.

We may clearly assume that $|A-A|\leq (2d-1)|A|$, since otherwise we already have the required conclusion. By the Pl\"unnecke--Ruzsa inequality, we then have $|A+A|\leq (2d-1)^2|A|$. Applying Lemma~\ref{lem:mudgal} with $c=(2d-1)^2$, we get parallel lines $l_1,\ldots,l_r$ and constants $0<\sigma\leq 1/2$ and $C>0$ such that 
$$|A \cap l_1| \geq \cdots \geq |A \cap l_r| \geq |A \cap l_1|^{1/2} \geq C^{-1}n^\sigma$$
and
$$|A \setminus (l_1 \cup l_2 \cup \cdots \cup l_r)| < Ccn^{1-\sigma},$$
where $n=|A|$. Since $|A\cap l_i|\geq C^{-1}n^\sigma$ for each $i$, we have $n=|A|\geq rC^{-1}n^\sigma$ or $r\leq Cn^{1-\sigma}$. Let $A'=A\cap (l_1\cup \cdots \cup l_r)$ and $S=A\setminus A'$, so that $|S|<Ccn^{1-\sigma}$. If $\dim(A')=d_1<d$, then, by our induction hypothesis, for $|A|$ sufficiently large, 
$$|A'-A'|\geq \paren{2d_1 - 2+\frac{1}{d_1-1}}|A'| - (2d_1^2-4d_1+3).$$
There are $a_1,\ldots,a_{d-d_1}\in S$ such that $\dim(A'\cup\set{a_1,\ldots,a_{d-d_1}})=d$. This implies that $a_1,\ldots,a_{d-d_1}$ lie outside the affine span of $A'$, so the sets
$$A'-A',A'-a_1,\ldots,A'-a_{d-d_1},a_1-A',\ldots,a_{d-d_1}-A'$$
are pairwise disjoint. Thus,
\begin{align*}
|A-A| &\geq |A'-A'|+\sum_{i=1}^{d-d_1}(|A'-a_i|+|a_i-A'|)\\
&\geq \paren{2d_1-2+\frac{1}{d_1-1}}|A'|-(2d_1^2-4d_1+3)+2(d-d_1)|A'|\\
&\geq \paren{2d-2+\frac{1}{d_1-1}}(|A|-|S|)-(2d_1^2-4d_1+3)\\
&\geq \paren{2d-2+\frac{1}{d-1}}|A|
\end{align*}
for $|A|\geq n_0$  sufficiently large. Thus, we may assume that $\dim(A')=d$.

For $n_0$ sufficiently large, we may assume that each line $l_i$ intersects $A'$ in at least $4d$ points. Let $H$ be a major hyperplane of $A'$ with respect to $l_1$ and let $H_1=H,H_2,\ldots,H_r$ be the translates of $H$ covering $A'$ in the natural order. Fix $0<\epsilon<1/(4d+1)(d-1)$. If we are in the case of Lemma~\ref{lem:lines} where
$$|A'-A'|\geq \paren{2d - 2+\frac{1}{d-1}+\epsilon}|A'|-C_d,$$
then, since $|S|=O(|A|^{1-\sigma})$ is sublinear, for $|A|$ sufficiently large,
\begin{align*}
    |A-A| &\geq |A'-A'|\\
    &\geq \paren{2d - 2+\frac{1}{d-1}+\epsilon}|A'|-C_d\\
    &\geq \paren{2d - 2+\frac{1}{d-1}}|A|.
\end{align*}
Thus, we may assume that $r=2$ and 
\[|A'-A'| \geq (2d-2)|A'|+\frac{2}{d-1}|A_1'|-(2d^2-4d+3).\]

Let $A_1'=A'\cap H_1$ and $A_2'=A'\cap H_2$. If $S\not\subseteq H_1\cup H_2$, then there is a point $x\in S$ not lying on the hyperplanes $H_1,H_2$. But then $x-A_i',A_i'-x,A'-A'$ are pairwise disjoint for some $i\in \set{1,2}$ and so, since $|A'_1| \geq |A'_2|$, 
\begin{align*}
    |A-A| &\geq |A'-A'|+2|A_2'|\\
    &\geq (2d-2)|A'|+\frac{2}{d-1}|A_1'|-(2d^2-4d+3)+2|A_2'|\\
    &\geq \paren{2d-2+\frac{2}{d-1}}|A'|-(2d^2-4d+3)\\
    &\geq \paren{2d-2+\frac{1}{d-1}}|A|.
\end{align*}
We may therefore assume that $S\subseteq H_1\cup H_2$.

Let $A_1=A\cap H_1$ and $A_2=A\cap H_2$. Let $H'$ be a major hyperplane of $A$ with respect to $l_1$ 
(possibly equal to $H$) and $H_1'=H',H_2',\ldots,H_s'$ be the translates of $H'$ covering $A$, ordered naturally. Let $B_i=A\cap H_i'$ for $i=1,\ldots,s$. Since $H_1,H_2$ are both supporting hyperplanes of $A$, we must have $|B_1|\geq \max(|A_1|,|A_2|)\geq |A|/2 >|S|$, so $B_1$ must contain at least one point of $A'$. Hence, $B_1$ contains one of the lines $l_i\cap A$, each of which has at least 2 points, and so $\dim (B_1)=d-1$.  

Suppose $s=2$. The number of lines parallel to $l_1$ intersecting $B_1$ is at most $r+|S|=O(|A|^{1-\sigma})$, 
which is smaller than $|B_1|/4$. Thus, for $n_0$ sufficiently large, by both cases of Lemma~\ref{lem:2planes},
$$|A-A|\geq \paren{2d-2+\frac{1}{d-1}}|A|-(2d^2-4d+3).$$

We may therefore assume that $s>2$. Let $B=A\setminus B_s$, noting that $|B|\geq |A|/2$ and $\dim(B)=d$. By our induction hypothesis,
$$|B-B|\geq \paren{2d-2+\frac{1}{d-1}}|B|-D.$$
Thus, again using Ruzsa's inequality~\eqref{eqn:ruzsa},
\begin{align*}
    |A-A| &\geq |B-B|+2|B_1-B_s|\\
    &\geq \paren{2d-2+\frac{1}{d-1}}|B|-D+2|B_1|+(2d-2)|B_s|-d(d-1)\\
    &\geq \paren{2d-2+\frac{1}{d-1}}|A|+\paren{2-\frac{1}{d-1}}|B_1|-d(d-1)-D\\
    &\geq \paren{2d-2+\frac{1}{d-1}}|A|+\paren{1-\frac{1}{2(d-1)}}|A|-d(d-1)-D\\
    &\geq \paren{2d-2+\frac{1}{d-1}}|A|-D+|A|/3,
\end{align*}
where the last inequality holds if $|A|/6\geq n_0/6\geq d(d-1)$.
\end{proof}

\section{Concluding remarks}

By carefully analysing our proof of Theorem~\ref{thm:main}, it is possible to deduce some structural properties of large sets $A \subset \RR^d$ with $\dim(A) = d$ and
$$|A-A|\leq \paren{2d-2+\frac{1}{d-1}}|A|+o(|A|).$$
In particular, such sets can be covered by two parallel hyperplanes $H_1$ and $H_2$, where, writing 
$A_1 = A \cap H_1$ and $A_2 = A \cap H_2$, we can assume that $A_1$ and $A_2$ have roughly the same size, differing by $o(|A|)$. We can also assume that $\dim(A_1) = d-1$ and that $A_1$ can be covered by $d-1$ parallel lines $l_1,\ldots,l_{d-1}$, where the sets $A_1\cap l_i$ all have approximately equal size, again up to $o(|A|)$.

In practice, $H_1$ will be a major hyperplane of $A$ with respect to $l_1$, which, we recall, means that it is parallel to $l_1$, it is supporting, in the sense that all points of $A$ lie either on or on one side of it, and $|H_1 \cap A|$ is as large as possible. Knowing this allows us to also deduce that $\dim (A_2)=d-1$. Indeed, it must be the case that the affine span of $A_2$ is parallel to $l_1$, since otherwise $|A_1 - A_2|$ would be too large. But then, if $\dim(A_2)<d-1$, there is a supporting hyperplane through $A_2$ and one of the $A_1\cap l_i$ which contains more points than  $H_1$, contradicting the fact that $H_1$ is a major hyperplane. Since $|A_1|$ and $|A_2|$ differ by $o(|A|)$, this then allows us to argue that $A_2$ is also covered by $d-1$ lines parallel to $l_1$ of approximately equal size.

In fact, we can deduce the very same structural properties for large sets $A\subset \RR^d$ with $\dim(A)=d$ and
$$|A-A|\leq \paren{2d-2+\frac{1}{d-1}+\epsilon}|A|+o(|A|)$$
for some $\epsilon>0$, giving a difference version of Stanchescu's result  about the structure of $d$-dimensional subsets of $\RR^d$ with doubling constant smaller than $d+ 4/3$, which we stated as  Theorem~\ref{thm:Stan}. It would be interesting to determine the maximum value of $\epsilon$ for which this continues to hold.

Unfortunately, our methods tell us very little about how $A_1$ and $A_2$ are related, though we suspect that $A_2$ should be close to a translate of $-A_1$. Proving this, which will likely require a better understanding of when Ruzsa's inequality~\eqref{eqn:ruzsa} is tight, may then lead to a determination of the exact structure of $d$-dimensional subsets $A$ of $\RR^d$ with $|A-A|$ as small as possible in terms of $|A|$, a problem that was already solved for $d = 2$ and $3$ by Stanchescu~\cite{S98}.

\vspace{3mm}
\noindent
{\bf Note added.} Shortly after completing this paper, we learned from Akshat Mudgal that he had independently proved an asymptotic version of Conjecture 1.1. We refer the reader to his paper~\cite{M21} for further details.

\end{document}